\documentclass{amsart}
\newcommand{\cA}{\mathcal{A}}

\newcommand{\F}{\mathbb{F}}

\newcommand{\bF}{\mathbb{F}}

\newcommand{\bN}{\mathbb{N}}

\newcommand{\bR}{\mathbb{R}}

\newcommand{\bZ}{\mathbb{Z}}
\newtheorem{atheorem}{Theorem}

\newcommand{\R}{\mathbb{R}}

\newcommand{\Z}{\mathbb{Z}}

\newcommand{\m}{\to}
\usepackage[dvipsnames,svgnames,x11names,hyperref]{xcolor}
\usepackage{url,graphicx,verbatim,amssymb,enumerate,stmaryrd}
\usepackage[pagebackref,colorlinks,citecolor=blue,linkcolor=blue,urlcolor=blue,filecolor=blue]{hyperref}
\usepackage{microtype}
\usepackage[all]{xy}
\usepackage[hmargin=3cm,vmargin=3cm]{geometry}
\newtheorem{theorem}{Theorem}[section]
\newtheorem{lemma}[theorem]{Lemma}
\newtheorem{proposition}[theorem]{Proposition}
\newtheorem{corollary}[theorem]{Corollary}

\theoremstyle{definition}
\newtheorem{definition}[theorem]{Definition}

\newtheorem{remark}[theorem]{Remark}

\newcommand{\mr}[1]{{\rm #1}}
\newcommand{\fS}{\mathfrak{S}}

\newcommand{\coker}{\mr{coker}}
\title[Sharper periodicity and stabilization maps]{Sharper periodicity and stabilization maps for configuration spaces of closed manifolds}
\author{Alexander Kupers}
\thanks{Alexander Kupers is supported by a William R. Hewlett Stanford Graduate Fellowship, Department of Mathematics, Stanford University, and was partially supported by NSF grant DMS-1105058.}
\author{Jeremy Miller}
\date{\today}

\begin{document}
\begin{abstract}
In this note we study the homology of configuration spaces of closed manifolds. We sharpen the eventual periodicity results of Nagpal and Cantero-Palmer, prove integral homological stability for configuration spaces of odd-dimensional manifolds and introduce a stabilization map on the homology with $\Z[1/2]$-coefficients for configuration spaces of odd-dimensional manifolds.
\end{abstract}
\maketitle

\section{Introduction}
Let $F_k(M)$ denote the configuration space of $k$ ordered particles in a (topological) manifold $M$ and let $C_k(M)$ denote the unordered configuration space. That is,
\[F_k(M)=\{(m_1,\ldots,m_k) \in M^k\, |\, m_i \neq m_j \text{ for } i\neq j \} \qquad \text{and} \qquad C_k(M)=F_k(M)/\fS_k\] with $\fS_k$ the symmetric group on $k$ letters acting by permuting the terms. When $M$ is open, there are maps $t:C_k(M) \m C_{k+1}(M)$ which can be thought of as bringing a point in from infinity. These maps induce isomorphisms on homology in a range tending to infinity with $k$ \cite{Mc1}.
 
When $M$ is closed, the homology does not stabilize. For example, $H_1(C_k(S^2)) \cong \Z/(2k-2)\Z$ \cite{FV}. However, if you restrict what coefficients you consider, there are homological stability patterns for configuration spaces of particles in closed manifolds \cite{BCT, FT, Ch, RW, BMi,  Kn, federicomartin, Nag}. One of these patterns is \emph{eventual periodicity}, which is the content of the first part of the following theorem.
 
\begin{atheorem} \label{TheoremA}
Let $M$ be a connected manifold of finite type and $m \in \bN$. \begin{enumerate}[(i)]
\item If $\dim M$ is even, for all $d$ dividing $2m$ and all $i \leq f(k,M,\Z/d\bZ)$, we have that \[H_i(C_k(M);\bZ/d\bZ) \cong H_i(C_{k+m}(M);\bZ/d\bZ)\]
\item If $\dim M$ is odd, for all $i \leq f(k,M,\Z)$ we have that \[H_i(C_k(M);\bZ) \cong H_i(C_{k+1}(M);\bZ)\]
\end{enumerate}
Here $f(k,M,R) \geq k/2$ is the function in Theorem \ref{openConfig}.
\end{atheorem}
 
Specializing Theorem \ref{TheoremA} to the case $m=1$ recovers the statement that the groups $H_i(C_k(M);\F_2)$ stabilize, a result appearing in \cite{RW} and which can also be deduced from the calculation of \cite{BCT}. Regarding the second part, it was previously known that when $M$ is odd-dimensional, the homology groups stabilize when working with field coefficients \cite{BCT,RW} and $\bZ[1/2]$-coefficients \cite{federicomartin}. It is interesting that although the two parts of Theorem \ref{TheoremA} seem different, their proofs are extremely similar.

The first part of Theorem \ref{TheoremA} is an improvement on two recent results:
\begin{itemize}
\item As Theorem F of \cite{Nag}, Nagpal proved the following.
\begin{theorem}[Nagpal]
There are constants $N_i$ such that for all finite type connected orientable manifolds $M$, $H_i(C_k(M);\F_p) \cong H_i(C_{k+p^{(i+3)(2i+2)}}(M);\F_p)$ for all $k \geq N_i$.
\end{theorem}
Nagpal's approach is more general and goes via the cohomology of the symmetric group with coefficients in an FI-module (see \cite{churchellenbergfarb}) over $\F_p$: he proved that if $G$ is a finitely generated FI-module over $\F_p$, the groups $H^j(B \fS_k;G_k)$ are eventually periodic with period a power of $p$. By applying this to the special case $G_k = H^i(F_k(M);\bF_p)$ and using the Serre spectral sequence, Nagpal proved that $H^j(C_k(M);\F_p)$ is eventually periodic. He also gave examples of FI-modules $G$ over $\F_p$ such that the periods of $H^j(B\fS_k;G_k)$ are arbitrarily large powers of $p$.

Theorem \ref{TheoremA} implies that the groups $H^j(C_k(M);\F_p)$ are eventually periodic with period at most $p$. We suspect that $H^j(B \fS_k; H^i(F_k(M);\F_p))$ also has period $p$ as opposed to a higher power of $p$. It would be interesting to know if that is true, and if so, what property of the FI-modules $H^i(F_k(M);\F_p)$ implies that $H^j(C_k(M);\F_p)$ has period $p$.
\item As Corollary E of \cite{federicomartin}, Cantero and Palmer proved the following.
\begin{theorem}[Cantero-Palmer]
Fix a prime $p$ and let $M$ be a connected smooth finite type manifold such that $\chi(M) \neq 0$. Let $L(M)$ be one plus the $p$-adic valuation of $\chi(M)$. Then $H_i(C_k(M);\F_p) \cong H_i(C_{k+p^{L(M)}}(M);\F_p)$ for all $i \leq f(k,M,\F_p)$.
\end{theorem}
Cantero and Palmer's result differs from Nagpal's in that it applies to non-orientable manifolds, but does not apply to manifolds with $\chi(M) = 0$. Furthermore, their period depends on the manifold, while Nagpal's depends on the homological degree. Cantero and Palmer's result is phrased for smooth manifolds and uses several smooth techniques, however their proof can probably be made to work for topological manifolds as well. On the other hand, we should point out that Nagpal's results also give information about the twisted homology of symmetric groups while the methods of this paper and of Cantero and Palmer do not.
\end{itemize}
 
We do not know of any natural maps inducing the isomorphisms of Theorem \ref{TheoremA}. In particular, we are unable to prove that the cohomological versions of the isomorphisms of Theorem \ref{TheoremA} respect the cup product. The second goal of this note is to describe a new way of producing maps between the homology of configuration spaces of particles in closed manifolds. Fix a point in $m_0$. One cannot define a map from $C_k(M)$ to $C_{k+1}(M)$ by simply adding a particle at $m_0$ to the configuration as there may already be a particle there. However, this procedure can be made to work when $M$ is odd-dimensional after inverting 2:
\begin{atheorem}
Let $M$ be odd-dimensional and connected. There is a natural split injection
\[\sigma: H_i(C_k(M);\Z[1/2]) \m H_i(C_{k+1}(M);\Z[1/2])\]
which is an isomorphism for $i \leq f(k,M,\Z[1/2])$.
\label{TheoremB}
\end{atheorem}
 
The map $\sigma$ will be defined as a zig-zag of maps of spaces and hence the cohomological analogue of $\sigma$ respects cup products. As mentioned before, Cantero and Palmer proved homological stability with $\bZ[1/2]$ coefficients for odd dimensional manifolds in Theorem A of \cite{federicomartin}. However, since they use the scanning map of \cite{Mc1}, they can only define a map between homology groups in the stable range. We hope that considering maps like $\sigma$ will have applications to other contexts where there are no obvious stabilization maps.
 
In our proofs, we shall always assume that the dimension of $M$ is at least $2$. Theorem \ref{TheoremA} and Theorem \ref{TheoremB} are also true if $M$ is one-dimensional since $C_k(\R) \simeq *$ and $C_k(S^1) \simeq S^1$.
 
\subsection*{Acknowledgments} We would like to thank Federico Cantero, James McClure, Martin Palmer and the referee for helpful comments.

\section{Periodicity}
In this section we prove Theorem \ref{TheoremA}. We start by recalling some algebraic structure of configuration spaces. Let $n$ be the dimension of the manifold $M$. Let $C(M) = \bigsqcup_k C_k(M)$. An embedding $\bigsqcup_i M_i \hookrightarrow M$ induces a map $\prod_i C(M_i) \m C(M)$. This structure induces an $E_d$-algebra structure on $C(N \times \R^d)$. We will define the stabilization map $t$ of \cite{Mc1} using this functoriality with respect to embeddings.

By definition, an open manifold is a manifold with no compact connected components. If $M$ is open, there is an embedding $\R^n \sqcup M \hookrightarrow M$ such that $M \hookrightarrow M$ is isotopic to the identity (see for example \cite{kupersmillerimprov}). This induces a map $C(\R^n) \times C(M) \m C(M)$. By picking a point in $C_1(\R^n)$ and restricting, this gives a map $t:C_k(M) \m C_{k+1}(M)$ called the \emph{stabilization map}. If $M$ is the interior of a manifold with finite handle decomposition, the stabilization map depends only on the choice of end of $M$. In the introduction we mentioned that $t$ induces an isomorphism in a range. Currently the best known results in this direction are the following \cite{Mc1,Se,Ch,RW,Kn,kupersmillerimprov,federicomartin}.

\begin{theorem} Let $M$ be an open connected $n$-dimensional manifold and $R$ be a ring. The stabilization map $t:C_k(M) \m C_{k+1}(M)$ induces an isomorphism on $H_i(-;R)$ for $i \leq f(k,M,R)$. The function $f(k,M,R)$ is called the stable range and the current best known bounds on it are: \[f(k,M,R)=\begin{cases}
k & \text{if 2 is a unit in $R$ and $dim(M) \geq 3$}\\
k & \text{if $R$ is a field of characteristic zero and $M$ is a non-orientable surface}\\
k-1 & \text{if $R$ is a field of characteristic zero and $M$ is an orientable surface}\\
k/2 & \text{otherwise} \\
\end{cases}\]
\label{openConfig}
\end{theorem}

Now suppose $M$ is closed and fix $m_0 \in M$. There are embeddings
\[\R^n \sqcup (M\backslash\{m_0\}) \hookrightarrow M\backslash\{m_0\} \qquad \text{and} \qquad (S^{n-1} \times \R) \sqcup (M\backslash\{m_0\}) \hookrightarrow M\backslash\{m_0\}\]
which give $C(M \backslash \{m_0\})$ the structure of an $E_1$-module over the $E_1$-algebras $C(\R^n)$ and $C(S^{n-1} \times \R)$. This induces maps
\begin{align*}H_*(C(\R^n)) \otimes H_*(C(M \backslash \{m_0\})) &\m H_*(C(M \backslash \{m_0\})) \\ H_*(C(S^{n-1} \times \R)) \otimes H_*(C(M \backslash \{m_0\})) &\m H_*(C(M \backslash \{m_0\}))\end{align*}

\noindent Given $b \in H_j(C(S^{n-1} \times \R))$ or $H_j(C(\R^n))$, let $t_b:H_i(C(M\backslash\{m_0\}) \m H_{i+j}(C(M\backslash\{m_0\})$ be the map induced by multiplication by the homology class $b$. Let $P$ denote the class of a point in $H_0(C_1(S^{n-1} \times \R))$ or $H_0(C_1(\R^n))$, so that $t_P$ is the map induced by the stabilization map $t$. On Page 321 of \cite{RW} (Page 15 of the arXiv-version), Randal-Williams makes the following observation.

\begin{proposition}[Randal-Williams]
There is a cofiber sequence
\[C_k(M) \m \Sigma^n C_{k-1}(M\backslash\{m_0\})_+ \m \Sigma C_k(M\backslash\{m_0\})_+\]
The map induced on homology $H_i(C_{k-1}(M\backslash\{m_0\})) \m H_{i+n-1}(C_k(M\backslash\{m_0\}))$ is given by $t_S$ with $S$ the image of the fundamental class of $S^{n-1}$ under the homotopy equivalence $S^{n-1} \m C_1(S^{n-1} \times \R)$.
\label{cofiber}
\end{proposition}

Let $b \in H_j(C_q(\R^n))$. The following diagram does not commute in general:
\[\xymatrix{H_i(C_{k-1}(M \backslash \{m_0\})) \ar[r]^{t_S} \ar[d]_{t_b} & H_{n-1+i}(C_{k}(M \backslash\{m_0\})) \ar[d]_{t_b} \\
H_{i+j}(C_{k-1+q}(M \backslash\{m_0\})) \ar[r]^{t_S} & H_{n-1+i+j}(C_{k+q}(M \backslash\{m_0\})) }\]
The failure of the diagram to commute is measured by the Browder operation, a homology operation for $E_n$-algebras, introduced in \cite{Br}. The following is an easy generalization of Lemma 9.2 of \cite{RW}.

\begin{proposition}
Let $P \in H_0(C_1(\R^n))$ be the class of a point and let $\phi$ denote the Browder operation. For any $b \in H_j(C_q(\R^n))$, we have that $t_S \circ t_b=t_b \circ t_S + t_{\phi(b,P)}$.

\end{proposition}

See Figure \ref{figure:diagramCommute} for a picture when $b$ is the class of $3$ points. The following corollary is immediate.
\begin{figure}[t]
\begin{center}\includegraphics[width=\textwidth]{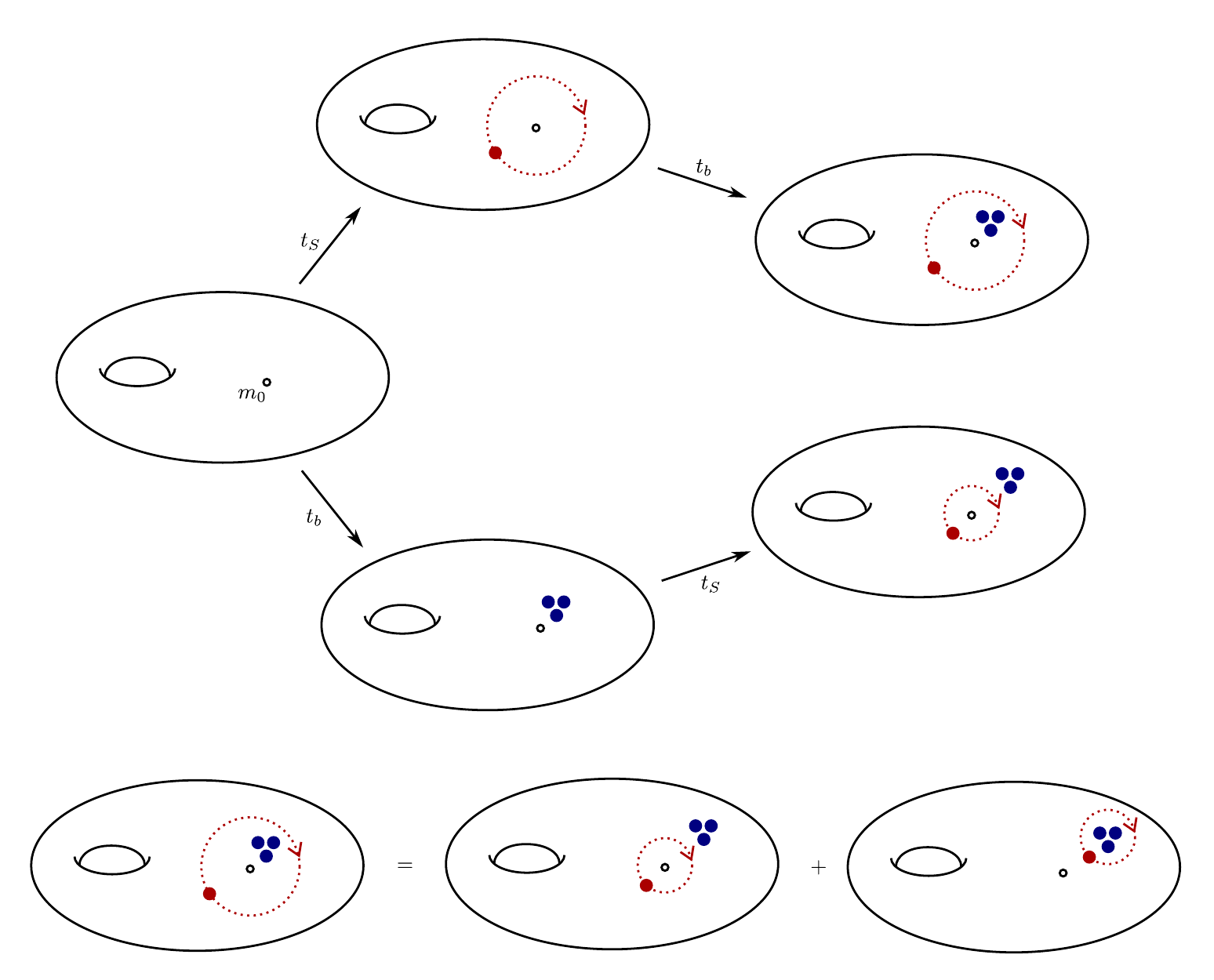}\end{center}
\caption{The difference between $t_S \circ t_b$ and $t_b \circ t_S$ for $b=P^3$.}
\label{figure:diagramCommute}
\end{figure}

\begin{corollary} \label{corCommute}
Let $R$ be a ring and let $b \in H_j(C_q(\R^n);R)$ be an element such that $\phi(P,b)=0$. Then the following diagram commutes:
\[\xymatrix{H_i(C_{k-1}(M\backslash \{m_0\};R)) \ar[r]^{t_S} \ar[d]_{t_b} & H_{n-1+i}(C_{k}(M\backslash\{m_0\});R) \ar[d]_{t_b} \\
H_{i+j}(C_{k-1+q}(M\backslash \{m_0\});R) \ar[r]^{t_S} & H_{n-1+i+j}(C_{k+q}(M \backslash \{m_0\});R) }\]
\end{corollary}

The following is implicit in Part III of \cite{CLM}.

\begin{proposition}[Cohen] If $n$ is even, then for any $m$ we have that $\phi(P,P^m)$ vanishes with $\Z/d\Z$ coefficients for $d$ dividing $2m$. If $n$ is odd, $\phi(P,P)$ vanishes with $\Z$ coefficients. Here the exponents indicate iterated Pontryagin product. \label{cohen}
\end{proposition}

\begin{proof}
Since $\phi$ is a derivation of the product in each variable, $\phi(P,P^m)= m P^{m-1}\phi(P,P)$. We have that $C_2(\R^n) \simeq \R P^{n-1}$ and $\phi(P,P) \in H_{n-1}(C_2(\bR^n))$ corresponds to twice the generator of $H_{n-1}(\R P^{n-1})$. If $n$ is odd, then twice the generator of $H_{n-1}(\R P^{n-1})$ is zero and hence so is $\phi(P,P)$. If $n$ is even, then $\phi(P,P^m)$ is $2m$ times another homology class and hence vanishes with $\Z/d\Z$ coefficients for $d$ dividing $2m$. \end{proof}

Randal-Williams used the second result in the case $m=1$ to prove stability for configuration spaces with field coefficients when $M$ is odd-dimensional and used the first result in the case $m=1$ to prove stability with $\Z/2\Z$ coefficients. We will use the first result to prove eventual periodicity and the second result to prove homological stability. Before we prove Theorem \ref{TheoremA}, we need the following algebraic lemma.

\begin{lemma}
Let $C_*$ be a bounded below chain complex of finitely generated free abelian groups and $p$ be a prime. Then the isomorphism type of the group $H_i(C_*  \otimes \Z/p^r \Z)$ is determined uniquely by the cardinality of the sets $H_j(C_* \otimes \Z/p^w\Z)$ for $w \leq r$ and $j \leq i$.

Similarly, the isomorphism type of the group $H_i(C_*)$ is determined uniquely by the cardinality of the sets $H_j(C_* \otimes \bZ/p^r \bZ)$ for all $p$ prime, $r \geq 1$ and $j \leq i$.
\label{algebralemma}
\end{lemma}

\begin{proof}We will prove a slightly stronger version of the first part, namely that these cardinalities not only determine uniquely the groups $H_i(C_* \otimes \bZ/p^r\bZ)$, but also the maps $H_i(C_* \otimes \bZ/p^{w}\bZ) \to H_i(C_* \otimes \bZ/p^{w-1}\bZ)$ for $w \leq r$. 

Let $A^{(r)}_*$ be the chain complex with $A_0=\Z/p^r \Z$ and all other groups equal to $0$. For $0< s < r$, let $A^{(s)}_*$ be the chain complex with $A^{(s)}_i=\Z/p^r \Z$ for $i=0$ or $1$, all other groups equal to $0$ and the map $A^{(s)}_1 \to A^{(s)}_0$ given by multiplication by $p^s$. Since $C_*$ is finitely generated in each degree, $C_* \otimes \Z/p^r\Z$ is quasi-isomorphic to a direct sum of a finite number of shifts of $A^{(s)}_*$. The proof is the same as that for Exercise 43 in Section 2.2 of \cite{hatcherbook}, working over $\bZ/p^r\bZ$ instead of $\bZ$. Let $a_i^{(s)}$ be non-negative integers such that \[C_* \otimes \Z/p^r\Z \simeq \bigoplus_i \left (\bigoplus_{s=1}^{r} (A^{(s)}_*[i])^{\oplus a_i^{(s)}} \right ) \]
with $[i]$ denoting a shift in homological degree by $i$. The $a_i^{(s)}$ depend on $p$ and $r$, as well as $i$ and $s$, but we are suppressing that from the notation. 

We claim that to prove the stronger version of the first part, it suffices to determine the $a_i^{(s)}$. Clearly these determine the homology groups $H_i(C_* \otimes \bZ/p^r\bZ)$. Note that both $C_* \otimes \bZ/p^r\bZ$ and $\bigoplus_i (\bigoplus_{s=1}^{r} (A^{(s)}_*[i])^{\oplus a_i^{(s)}})$ are free over $\bZ/p^r\bZ$ in each degree. Thus we can tensor both with $\bZ/p^w\bZ$ over $\bZ/p^r\bZ$ and apply the universal coefficient spectral sequence to obtain 
\[C_* \otimes \Z/p^w\Z \simeq \bigoplus_i \left(\bigoplus_{s=1}^{r} (A^{(s)}_*[i] \otimes \bZ/p^w\bZ)^{\oplus a_i^{(s)}}\right)\]
for all $w \leq r$. Under these isomorphisms the map $H_i(C_* \otimes \bZ/p^{w}\bZ) \to H_i(C_* \otimes \bZ/p^{w-1}\bZ)$ is induced by the reduction map $\bigoplus_i  (\bigoplus_{s=1}^{r} (A^{(s)}_*[i] \otimes \bZ/p^w\bZ)^{\oplus a_i^{(s)}} ) \to \bigoplus_i  (\bigoplus_{s=1}^{r} (A^{(s)}_*[i] \otimes \bZ/p^{w-1}\bZ)^{\oplus a_i^{(s)}} )$. Thus we can conclude that the numbers $a_i^{(s)}$ also determine this map.

Let $\#$ denote the cardinality of a set. Take $1 \leq w \leq r$. The homology of this chain complex is easy to compute and satisfies the following equation:
\[\sum_{s=1}^{r} a_i^{(s)} \cdot \log_p(\# H_0(A^{(s)}_*;\Z/p^w \Z)) =  \log_p(\# H_i(C_* \otimes \Z/ p^w \Z)) - \sum_{s=1}^{r} a_{i-1}^{(s)} \cdot \log_p(\# H_1(A^{(s)}_*;\Z/p^w \Z))\]
Using $\log_p(\# H_0(A^{(s)}_*;\Z/p^w \Z)) =\min(w,s)$ and $\log_p(\# H_1(A^{(s)}_*;\Z/p^w \Z)) =w-\min(w,s)$, this equation simplifies to
\[\sum_{s=1}^{r}  a_i^{(s)} \cdot \min(w,s)  =  \log_p(\# H_i(C_* \otimes \Z/ p^w \Z))  - \sum_{s=1}^{r}  a_{i-1}^{(s)} \cdot (w - \min(w,s))\] Since the $(r \times r)$-matrix with $(w,s)$th entry given by $\min(w,s)$ is invertible, the above system of linear equations on $a_i^{(s)}$ has a unique solution. That is, $a^{(s)}_{i}$ are completely determined by $a^{(1)}_{i-1}, \ldots a^{(r)}_{i-1}$ and the cardinality of $H_i(C_* \otimes \Z/p^w\Z)$ for $w \leq r$ . Since $C_*$ is a bounded below chain complex, there exists a $J \in \bZ$ such that $a_j^{(s)}=0$ for $j \leq J$.  That we can determine the $a_i^{(s)}$ now follows by induction on the statement that the cardinalities of $H_j(C_* \otimes \Z/p^w\Z)$ for $1 \leq w \leq r$ and $J \leq j \leq i$ determine the numbers $a_j^{(s)}$ for $1 \leq s \leq r$ and $J \leq j \leq i$. 
\\

The second part is a consequence of the first part as follows. It suffices to show that the abelian groups $H_i(C_* \otimes \bZ/p^r \bZ)$ and maps $H_i(C_* \otimes \bZ/p^{r+1} \bZ) \to H_i(C_* \otimes \bZ/p^r \bZ)$ for all primes $p$ and $r \geq 1$, determine the free and $p$-torsion summands of $H_i(C_*)$. These summands can be recovered from $H_i(C_*)\otimes \bZ_{p}^\wedge$, where $\bZ_{p}^\wedge$ is the group of $p$-adic integers. Note $\bZ_p^\wedge$ is flat over $\bZ$, as it is torsion-free, so $H_i(C_*)\otimes \bZ_{p}^\wedge \cong H_i(C_* \otimes \bZ_{p}^\wedge)$. The $p$-adics are the limit of $\bZ/p^r \bZ$ and the maps $C_i \otimes \bZ/p^{r+1} \bZ \to C_i \otimes \bZ/p^r \bZ$ are all surjective. By the Mittag-Leffler condition we conclude that $H_i(C_* \otimes \bZ_{p}^\wedge) = \lim_r H_i( C_* \otimes \bZ/p^r \bZ)$.\end{proof}


\begin{proof}[Proof of Theorem \ref{TheoremA}] If $M$ is open, this follows from Theorem \ref{openConfig} using the universal coefficient theorem, so we can assume $M$ is closed. We start with the proof of part (i). To prove the statement for $d \in \bN,$ it suffices to prove the statement for $q = p^r$ with $p$ prime since homology with coefficients in a direct sum of abelian groups is naturally isomorphic to the direct sum of the homology groups with coefficients in each summand. We will assume that $p$ is odd for ease of notation. For $p=2$, one modifies the proof simply by replacing $\Z/q\Z$ with $\Z/2 q\Z$. Consider the long exact sequence \[\cdots \to H_{i+1-n}(C_{k-1}(M\backslash\{m_0\});\bZ/q\bZ) \overset{t_S}{\m} H_{i}(C_k(M\backslash\{m_0\});\bZ/q\bZ) \m H_{i}(C_k(M);\bZ/q\bZ) \to \cdots \] associated to the cofiber sequence of Proposition \ref{cofiber}. We obtain that
\begin{align*}
\# H_i(C_k(M);\bZ/q\bZ) = \#&\coker [H_{i+1-n}(C_{k-1}(M\backslash\{m_0\});\bZ/q\bZ) \overset{t_S}{\m} H_{i}(C_k(M\backslash\{m_0\});\bZ/q\bZ)] \\ &\cdot \#\ker [H_{i-n}(C_{k-1}(M\backslash\{m_0\});\bZ/q\bZ) \overset{t_S}{\m} H_{i-1}(C_k(M\backslash\{m_0\});\bZ/q\bZ) ]
\end{align*}

Note that $t_{P^q}: H_*(C_k(M\backslash\{m_0\});\bZ/q\bZ) \m H_*(C_{k+q}(M\backslash\{m_0\});\bZ/q\bZ) $ is an isomorphism in the stable range since $t_{P^{q}}$ is an iteration of the map on homology considered in Theorem \ref{openConfig}. Since $\phi(P,P^{q})=0$, the following diagram commutes: \[\xymatrix{H_{i+1-n}(C_{k-1}(M\backslash\{m_0\});\bZ/q\bZ) \ar[r]^{t_S} \ar[d]_{t_{P^{q}}} & H_{i}(C_{k}(M\backslash\{m_0\});\bZ/q\bZ) \ar[d]_{t_{P^{q}}} \\
H_{i+1-n}(C_{k-1+q}(M\backslash\{m_0\});\bZ/q\bZ) \ar[r]^{t_S} & H_{i}(C_{k+q}(M\backslash\{m_0\});\bZ/q\bZ) }\]

Thus, in the stable range, $t_{P^{q}}$ induces isomorphisms \begin{align*}\coker& [H_{i+1-n}(C_{k-1}(M\backslash\{m_0\});\bZ/q\bZ) \overset{t_S}{\m} H_{i}(C_k(M\backslash\{m_0\});\bZ/q\bZ)] \\ &\cong \coker [H_{i+1-n}(C_{k-1+q}(M\backslash\{m_0\});\bZ/q\bZ) \overset{t_S}{\m} H_{i}(C_{k+q}(M\backslash\{m_0\});\bZ/q\bZ)]\end{align*}
\begin{align*}\ker& [H_{i-n}(C_{k-1}(M\backslash\{m_0\});\bZ/q\bZ) \overset{t_S}{\m} H_{i-1}(C_k(M\backslash\{m_0\});\bZ/q\bZ) ] \\ &\cong \ker [ H_{i-n}(C_{k-1+q}(M\backslash\{m_0\});\bZ/q\bZ) \overset{t_S}{\m} H_{i-1}(C_{k+q}(M\backslash\{m_0\});\bZ/q\bZ)] \end{align*} 

Therefore, these kernels and cokernels have the same number of elements when you increase the number of particles by $q$. We conclude that $H_i(C_k(M);\Z/q\Z)$ and $H_i(C_{k+q}(M);\Z/q\Z)$ have the same number of elements in the stable range. Lemma \ref{algebralemma} applies since configuration spaces of a closed finite type manifold are homotopy equivalent to finite CW complexes and we can use cellular chains. Using the first part of that lemma we conclude that $H_i(C_k(M);\Z/q\Z)$ and $H_i(C_{k+q}(M);\Z/q\Z)$ are isomorphic in the stable range.\\

For part (ii), we use the same argument as above with coefficients in $\bZ/q\bZ$ with $t_P$ instead of $t_{P^q}$, with $q=p^r$. The input is now that $\phi(P,P)=0$ for $\bZ/q \bZ$ coefficients and the result is that $H_i(C_k(M);\bZ/q\bZ)$ and $H_i(C_{k+1}(M);\bZ/q\bZ)$ are isomorphic in the stable range. By the second part of Lemma \ref{algebralemma} we conclude that $H_i(C_k(M);\bZ)$ and $H_i(C_{k+1}(M);\bZ)$ are isomorphic.\end{proof}

\begin{remark} Let $p$ be an odd prime and fix a non-negative integer $i$. Theorem A of \cite{federicomartin} implies that the groups $H_i(C_k(M);\Z/p \Z)$ are all isomorphic to each other provided that $p$ does not divide $2k - \chi(M)$. On the other hand, Theorem \ref{TheoremA} of this paper implies that the groups $H_i(C_k(M);\Z/p \Z)$ are all isomorphic to each other provided that $p$ \emph{does} divide $2k - \chi(M)$. Thus, there are at most two distinct possible isomorphism types for the groups $H_i(C_k(M);\Z/p\Z)$ in the stable range. Using more consequences of Theorem A of \cite{federicomartin}, we can generalize to say that there are at most $r+1$ distinct possible isomorphism types for the groups $H_i(C_k(M);\Z/p^r\Z)$ in the stable range. 

When $\chi(M)$ is even, combining Theorem A of \cite{federicomartin} and Theorem \ref{TheoremA} of this paper implies that there are at most $r$ distinct possible isomorphism types for the groups $H_i(C_k(M);\Z/2^r\Z)$ in the stable range. These upper bounds are achieved since $H_1(C_k(S^2)) \cong \Z/(2k-2)\Z$ for $k \geq 1$.
\end{remark}

\begin{remark}
If one were to improve the known homological stability ranges for configuration spaces of open manifolds, one would get an improved version of Theorem \ref{TheoremA}. In our proof we used the fact that $i \leq f(k,M,\Z/q\Z)$ implies $i+1-n\leq f(k-1,M,\Z/q\Z)$. This constraint prevents one from improving the stability range in Theorem \ref{TheoremA} beyond slope $n-1$. Currently there are no manifolds whose configuration spaces are known to have a homological stability slope higher than $n-1$ with any non-zero coefficients.
\end{remark}


\section{Stabilization maps}
In this section we prove Theorem \ref{TheoremB}. In order to define a stabilization map for closed manifolds, we will need to consider configuration spaces where particles can coincide. We first recall the definition of the symmetric product and bounded symmetric product of a space.

\begin{definition}
Given a space $X$, let $\mr{Sym}_k(X)$ denote the quotient of $X^k$ by the permutation action of the symmetric group $\fS_k$. Let $\mr{Sym}_k^{\leq j}(X)$ denote the subspace of $\mr{Sym}_k(X)$ where no more than $j$ particles are at the same location.
\end{definition}

A point in $\mr{Sym}_k^{\leq j}(X)$ is a multiset where no element is repeated more than $j$ times. We will think of $\mr{Sym}_k^{\leq j}(X)$ as the configuration space of points in $X$ which are labeled by their multiplicities. We have that $\mr{Sym}_k^{\leq 1}(X)=C_k(X)$ and $\mr{Sym}_k^{\leq j}(X)=\mr{Sym}_k(X)$ for $k \leq j$.

\begin{proposition}
Let $M$ be odd-dimensional. The inclusion map $\iota:C(M) \m \mr{Sym}^{\leq 2}(M)$ induces an isomorphism on homology with $\Z[1/2]$ coefficients.
\end{proposition}

\begin{proof} 
It suffices to prove that $\iota:C(M) \m \mr{Sym}^{\leq 2}(M)$ induces an isomorphism on cohomology with coefficients in any $\Z[1/2]$-module (see e.g. Proposition 3.3.11 of \cite{MayPonto}). Let $A$ be a $\Z[1/2]$-module and consider the Leray spectral sequence associated to $\iota$ with $A$ coefficients; see for example Section IV.6 of \cite{bredon}. Let $\cA$ denote the constant sheaf on $C(M)$ with value $A$ on connected sets. Let $R^j \iota_* \cA$ be the sheaf on $\mr{Sym}^{\leq 2}(M)$ associated to the presheaf $U \mapsto H^j(U \cap C(M);A)$. The stalks of $R^j \iota_* \cA$ around a point in $\mr{Sym}^{\leq 2}(M)$ are $H^j(\prod_i C_{\alpha_i}(\R^n);A)$ with $\alpha_i=1$ or $2$ being the multiplicities of the particles at the point of interest in $\mr{Sym}^{\leq 2}(M)$ (see \cite{Tot} for a similar discussion). Note that $C_1(\R^n)$ is contractible and $C_2(\R^n) \simeq \R P^{n-1}$. Since $n$ is odd, $C_2(\R^n)$ has the $A$ cohomology of a point. Thus, the higher derived pushforward sheaves are zero. This shows that $\iota:C(M) \m \mr{Sym}^{\leq 2}(M)$ induces an isomorphism on cohomology with $A$ coefficients.
\end{proof}

We can now define a stabilization map for configuration spaces of closed manifolds after inverting $2$.

\begin{definition} Fix a point $m_0 \in M$. Let $Q:C_k(M) \m \mr{Sym}^{\leq 2}_{k+1}(M)$ be the map that adds a particle at $m_0$. Let $\sigma: H_*(C_k(M);\Z[1/2]) \m H_*(C_{k+1}(M);\Z[1/2])$ be $(\iota_*)^{-1} \circ Q_*$.
\end{definition}

\begin{remark}An alternative construction uses the Vietoris-Begle theorem, which says a proper surjective map between locally compact spaces that has acyclic fibers is a homology equivalence; see for example Section V.6 of \cite{bredon}. Fix a closed disk $D$ around $m_0$ and let $Z \subset C_{k+1}(M)$ be the subspace such that there are either 1 or 2 particles in $D$, and if there are two particles they are distance $\geq 1/4$ apart in the standard metric of the disk. There is a map $\pi: Z \to C_k(M)$ which collapses the disk to $m_0$ and puts either 0 or 1 particles there. This is a proper surjective map between locally compact spaces with fibers homotopy equivalent to either $*$ or $\bR P^{n-1}$. Thus the Vietoris-Begle theorem implies that $\pi$ induces a homology isomorphism with $\bZ[1/2]$-coefficients. We have that $\sigma$ is also the composition $j_* \circ (\pi_*)^{-1}$, where $j: Z \hookrightarrow C_{k+1}(M)$ is the inclusion.\end{remark}

We now prove the stability portion of Theorem \ref{TheoremB} in the case that the manifold is open.
\begin{lemma} Let $M$ be a connected open odd-dimensional manifold. The map $\sigma: H_i(C_k(M);\Z[1/2]) \m H_i(C_{k+1}(M);\Z[1/2])$ is an isomorphism for $i \leq f(k,M,\Z[1/2])$. \label{open}
\end{lemma}

\begin{proof}
We have that $Q$ is homotopic to $\iota \circ t$. Since $t$ induces an isomorphism in the stable range and $\iota$ is a homology equivalence with $\Z[1/2]$ coefficients, the claim follows.
\end{proof}

\begin{remark}
This proof also shows that for odd-dimensional connected manifolds, the map on homology with $\Z[1/2]$ coefficients induced by the stabilization map does not depend on choice of end $M$.
\end{remark}

We now define an augmented semisimplicial space similar to those in Proposition 9.4 of \cite{RW} and Definition 5.3 of \cite{kupersmillertran}.
\begin{definition}We define an augmented semisimplicial space $X^{k,\leq j}_\bullet$ by taking as $i$-simplices for $i \geq -1$ the following
\[X^{k,\leq j}_i = \bigsqcup_{\vec{p} \in F_{i+1}(M\backslash\{m_0\}) } \mr{Sym}^{\leq j}_k(M \backslash \{p_0,\ldots, p_i\})\] where $F_i(M\backslash\{m_0\})$ denotes the ordered configuration space of $i$ particles in $M\backslash\{m_0\}$. The natural inclusions of $M \backslash \{p_0,\ldots, p_i\}$ into $M \backslash \{p_0, \ldots, \hat p_l, \ldots, p_i\}$ induce the face maps and augmentation.
\end{definition}

That is, a point in the space of $i$-simplicies of $X^{k,\leq j}_\bullet$ is a collection $i+1$ ordered punctures in $M\backslash\{m_0\}$ and an element of the bounded symmetric product of the punctured manifold. The $l$th face maps can be thought of as forgetting the $l$th puncture. The augmentation map forgets the last puncture. The only difference between these augmented semisimplicial spaces and those appearing in \cite{RW} and \cite{kupersmillertran} is that here we insist that no punctures can be at $m_0$. Note that $X^{k,\leq j}_{-1}=\mr{Sym}_k^{\leq j}(M)$ and so the augmentation induces a map $||X^{k,\leq j}_\bullet|| \m \mr{Sym}^{\leq j}_k(M)$.

\begin{lemma} The augmentation map induces a weak equivalence $||X^{k,\leq j}_\bullet|| \m \mr{Sym}^{\leq j}_k(M)$.
\end{lemma}

\begin{proof}
See Page 324 of \cite{RW} (Page 17 of the arXiv-version) or Proposition 5.8 of \cite{kupersmillertran}, which only require that $M\backslash\{m_0\}$ is an infinite set.
\end{proof}

\begin{theorem} Let $M$ be a connected odd-dimensional manifold. The map $\sigma: H_i(C_k(M);\Z[1/2]) \m H_i(C_{k+1}(M);\Z[1/2])$ is an isomorphism for $i \leq f(k,M,\Z[1/2])$.
\end{theorem}

\begin{proof} It suffices to show that $Q$ induces a homology equivalence in a range with $\Z[1/2]$ coefficients. The map $Q$ extends to a map $Q_\bullet : X^{k,\leq 1}_\bullet \m X^{k+1,\leq 2}_\bullet$ since there are no punctures at $m_0$. It suffices to show that this map is a levelwise homology equivalence in a range with $\Z[1/2]$ coefficients as the claim would then follow by considering the geometric realization spectral sequence. However, the manifolds appearing on all non-negative simplicial levels are open and so the claim follows by Lemma \ref{open}.
\end{proof}

To complete the proof of Theorem \ref{TheoremB}, we now prove that the map $Q$ is split injective on homology using Lemma 2.2 of \cite{Do}. To do this, we recall the definition of so-called \emph{transfer maps}. For $l \geq k$, there is a map
\[c: \mr{Sym}^{\leq j}_l(M)) \m \mr{Sym}_{l \choose k}(\mr{Sym}^{\leq j}_k(M))\] which records all $k$ element submultisets of a given multiset. For any space $X$, there is a natural map
\[a: H_*(\mr{Sym}_d(X)) \m H_*(X)\]
which is defined by viewing a chain in $\mr{Sym}_d(X)$ as $d$ chains in $X$ and adding the $d$ chains together.

\begin{definition}
For $l \geq k$, let $\tau^j_{k,l}: H_*(\mr{Sym}^{\leq j}_l(M)) \m H_*(\mr{Sym}^{\leq j}_k(M))$ be $a \circ c_*$.
\end{definition}

\begin{theorem} Let $M$ be an odd-dimensional manifold. The map $\sigma: H_*(C_k(M);\Z[1/2]) \m H_*(C_{k+1}(M);\Z[1/2])$ is split injective in all degrees.
\end{theorem}

\begin{proof}
By Lemma 2.2 of \cite{Do}, it suffices to check that $\tau^1_{k,l} \circ \sigma_{l-1} = \tau^1_{k,l-1} + \sigma_{k-1} \circ \tau^1_{k-1,l-1}$. Since $\iota_* \circ \tau^1_{k,l} = \tau^2_{k,l} \circ \iota_*$, it suffices to show that $\tau^2_{k,l} \circ (Q_{l-1})_* = \iota_* \circ \tau^1_{k,l-1} + (Q_{k-1})_*\circ \tau^1_{k-1,l-1}$. Here the subscripts on the maps $Q$ and $\sigma$ indicate the number of particles in the domain. This is the statement that if one starts with a multiset, adds an element and then deletes a collection of elements, this collection either contains the added element (first summand) or it does not (second summand). \end{proof}

\begin{remark}
The fact that $\sigma$ exists seems to be related to the fact that odd-dimensional spheres are $H$-spaces after inverting $2$, which indicates that $C(\R^n)$ is slightly more commutative than $E_n$ when $n$ is odd. The arguments of this section were inspired by the proof of Proposition 3.1 of the first arXiv version of \cite{BMi}. The proof of that proposition has a mistake, but works when $M$ is parallelizable. Applying the scanning map to the constructions of this section yields the constructions appearing in the proof of that proposition. \end{remark}

\bibliographystyle{amsalpha}
\bibliography{periodicity}
\end{document}